%% file: Manuscript.tex
\documentclass[10pt]{amsart}
\usepackage{amssymb,latexsym,amsmath,amsthm,enumitem,mathrsfs,geometry}
\usepackage{fullpage}
\usepackage{fancyhdr}
\usepackage{comment}
\usepackage{mathrsfs}
\usepackage{hyperref}
\usepackage{lineno}
\usepackage{diagbox}
\usepackage{array}
\usepackage{tikz}
\usetikzlibrary{arrows}
\usetikzlibrary{positioning}
\usepackage{float}
\input{format}
\newtheorem{remark}{Remark}
\usepackage[titletoc,title]{appendix}

\newcommand{\Q}{\mathbb{Q}}
\newcommand{\q}{\mathfrak{P}}
\newcommand{\p}{\mathfrak{p}}

\newcommand{\N}{\mathbb{N}}

\newtheorem{Definition}{Definition}

\def\P{\mathfrak{P}}
\def\p{\mathfrak{p}}

\newcommand{\kommentar}[1]{}
\newtheorem{theorem}{Theorem}%[section]

\newtheorem{hypothesis}{Hypothesis}

\newtheorem*{remark*}{Remark}

\newtheorem{Proposition}{Proposition}[section]

\definecolor{pink}{rgb}{1,.2,.6}
\definecolor{orange}{rgb}{0.7,0.3,0}
\definecolor{blue}{rgb}{.2,.6,.75}
\definecolor{green}{rgb}{.4,.7,.4}
\definecolor{purple}{RGB}{127,0,255}

\begin{document}
\numberwithin{equation}{section}

\title{Consecutive pure fields of the form  $\mathbb{Q}\left(\sqrt[l]{a}\right)$ with large class numbers}

 \author[Das]{Jishu Das}
 \address{Department of Applied Sciences and Humanities, Mathematics Section, 
 COEP Technological University, Wellesley Road, Shivajinagar, Pune, Maharashtra 411005, India.}
 \email{dasj.maths@coeptech.ac.in}
\author[Krishnamoorthy]{Srilakshmi Krishnamoorthy}
\address{Indian Institute of Science Education and Research Thiruvananthapuram, Maruthamala PO, Vithura,
Thiruvananthapuram - 695551,
 Kerala, India.}
 \email{srilakshmi@iisertvm.ac.in}


\date{\today}

\begin{abstract} 
Let  $l$ be a rational prime greater than or equal to $3$ and $k$ be a given positive integer. Under a conjecture due to  Langlands and an assumption on upper bound for the regulator of fields of the form $\mathbb{Q}\left(\sqrt[l]a\right)$,  we prove that there are atleast $x^{1/l-o(1)} $ integers $1\leq d\leq x$ such that the consecutive pure fields of the form $\mathbb{Q}\left(\sqrt[l]{d+1}\right),  \dots ,\mathbb{Q}\left(\sqrt[l]{d+k}\right) $ have arbitrary large class numbers. 
\end{abstract}
\maketitle

\section{Introduction}
The class number of an algebraic number field plays a vital role in number theory.
An important problem concerning class number of a number field is to have an understanding of the size of the class number.  Let $h_{K}$ denote the class number of the number field $K.$ Let $d$ be the fundamental discriminant of the number field $\Q(\sqrt{d})$. Assuming the Generalized Riemann Hypothesis (GRH), Littlewood \cite{LW} proves
\begin{equation}\label{LE1}
h_{\Q(\sqrt{-d})}\leq \left( \frac{2e^\gamma}{\pi} +\text{o}(1) \right) \sqrt{|d|}\log\log(|d|), 
\end{equation}
where $\gamma$ denotes the  Euler-Mascheroni constant. Under the same hypothesis, Littlewood shows the existence of  infinitely many fundamental
discriminants  $d$ such that 
$$h_{\Q(\sqrt{-d})}\geq \left( \frac{2e^\gamma}{\pi} +\text{o}(1) \right) \sqrt{|d|}\log\log(|d|). $$
For positive discriminants, one can show  that the  bound 
\begin{equation}\label{LE2}
h_{\Q(\sqrt{d})}\leq \left( 4e^\gamma +\text{o}(1) \right) \sqrt{d} \frac{\log\log(d)}{\log d}, 
\end{equation}
holds under GRH.  Montgomery and  Weinberger \cite{Mon} show that this bound cannot be improved for real quadratic fields,
apart from the value of the constant and an unconditional result without GRH. 
They prove that there exist infinitely
many real quadratic fields 
$\Q(\sqrt{d})$
 such that 
\begin{equation}\label{MW}
h_{\Q(\sqrt{d})}\gg  \sqrt{d} \frac{\log\log(d)}{\log d}.
\end{equation}
Lamzouri \cite[Theorem 1.2]{Lam} improved the result by Montgomery and  Weinberger by showing the following. Let $x$ be large.
There are at least $x
^{\frac{1}{2}-\frac{1}{\log \log x}}$ and at most $x^{\frac{1}{2}+\text{o}(1)}$
real quadratic fields $\Q(\sqrt{d})$ with discriminant
$d \leq x$ such that
\begin{equation}\label{Lamz}
  h_{\Q(\sqrt{d})}\geq \left( 2 e^\gamma +\text{o}(1) \right) \sqrt{d} \frac{\log\log(d)}{\log d}.  
\end{equation}

Assuming GRH and Artin's conjecture for Artin L-functions, Duke \cite[Theorem 1]{Dukecubic} proves an analogue lower bound like equation \eqref{MW} for more general number fields. To be precise, Duke considers the set $\mathcal{K}_n$ consisting of totally real number fields of degree $n $ whose Galois closures have $S_n$ as their Galois group. \cite[Theorem 1]{Dukecubic} states that  
 there is a constant $c>0$ depending only on $n$ such that there exist  $K\in \mathcal{K}_n$ with
 arbitrarily large discriminant $d$ for which $$h_{K}>c\sqrt{d} \left(\frac{\log\log d}{\log d}\right)^{n-1}.$$
  Daileda \cite[Theorem 1]{Daileda} proves 
 an unconditional version of Duke's result for $\mathcal{K}_3.$
 He further \cite[Theorem 3]{Daileda} shows that there is an absolute constant $c>0$ such that 
 $$h_{\Q(\sqrt[3]{d})}\geq  c\sqrt{|d|}  \frac{\log\log |d|}{\log |d|}.$$
  
In 2023, Cherubini, Fazzari, Granville, Kala and Yatsyna \cite{CFGKY} prove that for a given positive integer $k,$ there are at 
 least   $x^{1/2-\text{o}(1)}$ integers $1\leq d\leq x$ such that for all $j=1,\dots,k,$ 
$$h_{\Q(\sqrt{d+j})}\gg_k\frac{\sqrt{d}}{\log d}\log \log d.$$
To generalize the above, recently Byeon and Yhee \cite{DBDY} show the following: Given a positive integer $k,$ there are at least $x^{1/3-\text{o}(1)}$ integers $1 \leq d\leq x$ such that the consecutive pure cubic fields $\Q(\sqrt[3]{d+1})$, \dots, $\Q(\sqrt[3]{d+k})$ have arbitrary large class number.   To be precise, we have the following theorem.
\begin{theorem}[{{\cite[Theorem 1.1]{DBDY}}}]\label{DBDY}
    Let $k$ be a fixed positive integer. There are atleast $x^{1/3-\text{o}(1)}$ integers $1 \leq d\leq x$ such that the class number $h_{\Q(\sqrt[3]{d+j})}$ of the pure cubic field $\Q(\sqrt[3]{d+j})$ satisfy $$h_{\Q(\sqrt[3]{d+j})}\gg_k\frac{\sqrt d}{\log d}\log \log d,$$
    for all $j=1,\dots,k.$
\end{theorem}
Let $l$ be a given prime throughout the article.  We wish to answer the following in this article: Given a positive integer $k$ and an prime $l>3,$ do there exist at least $x^{1/l-\text{o}(1)}$ integers $1\leq d\leq x$ such that the pure fields $\Q(\sqrt[l]{d+1})$, $\Q(\sqrt[l]{d+2})$, \dots, $\Q(\sqrt[l]{d+k})$ have arbitrary large class numbers?  We start with a hypothesis that will be useful. 
\begin{hypothesis}\label{hypothesis 1}
Let $a=n^l+r$ where $n, r$ are positive integers and $r|ln^l $ with $l>3.$   Let $R_{\Q\left(\sqrt[l]a\right)}$ and $D_{\Q\left(\sqrt[l]a\right)}$ denote the regulator and the absolute value of the discriminant of the field $\Q(\sqrt[l]a ). $  Then $$R_{\Q\left(\sqrt[l]a\right)}=\text{o}\left( \sqrt{D_{\Q\left(\sqrt[l]a\right)}}\log\log D_{\Q\left(\sqrt[l]a\right)}\right).$$
\end{hypothesis}
    In Section \ref{sec 2.1}, we briefly discuss Hypothesis \ref{hypothesis 1} with a heuristic argument (see Section \ref{sec 2.1} and Remark \ref{Nechyp}), along with some computational evidence in Section \ref{comevi}. 

Let $\pi$ be a  finite dimensional complex representation of $\text{Gal}(\Q(\omega,\sqrt[l]{b})/\Q),$ 
 where $\omega$ is a primitive $l$-th root of unity and $b $ is a $l$-th powerfree positive integer. 
 By Langlands conjecture (\cite[Conjecture 1.8.1]{Bump}), there is an automorphic representation $\tilde{\pi}_{D}$ with conductor $D$ such that $$ L(s,\Q(\omega,\sqrt[l]{b})/\Q,\pi)=L(s,\Q(\omega,\sqrt[l]{b})/\Q,\tilde{\pi}_{D}),$$ where $D$ denotes the absolute value of the discriminant of $\Q(\sqrt[l]{b}),$ and   $L(s,\Q(\omega,\sqrt[l]{b})/\Q,\pi)$, $L(s,\Q(\omega,\sqrt[l]{b})/\Q, \tilde{\pi}_{D})$ denote the Artin L-function associated to $\pi$ and $\tilde{\pi}_{D}$ respectively. 
The following holds under the Langlands conjecture(\cite[Conjecture 1.8.1]{Bump}).
\begin{theorem}\label{main theorem}
    Let $k$ be a fixed positive integer and $l$ be a rational prime greater than or equal to $3$.  Let $D_j$ be the absolute value of the discriminant of the
pure field $\mathbb{Q}\left(\sqrt[l]{\Delta(m)+j} \right)$ (see Section \ref{two} for $\Delta(m)$). Let $\omega$ be a primitive $l$-th root of unity and $0<\epsilon<1$.  Suppose that 
$$
\sum_{ \substack{p\leq (\log D_j)^\epsilon, \\ p\equiv \,1(\text{mod}\, \,l) \\\left( \frac{D_j}{\pi}\right)_l=1} }  \frac{l-1}{p}+ \sum_{ \substack{p\leq (\log D_j)^\epsilon, \\ p\equiv 1 \,(\text{mod}\, \, l) \\\left( \frac{D_j}{\pi}\right)_l=1}}' \   \frac{\lambda(p)}{p}     \gg_k \log\log  \log x ,
$$
where $\lambda(p)$ corresponds to the Dirichlet series given by \eqref{L(s,pi)} for the Artin L-function $$L\left(s,\mathbb{Q}\left(\omega,\sqrt[l]{\Delta(m)+j} \right)/\Q, \tilde{\pi}_{d_{\mathbb{Q}\left(\sqrt[l]{\Delta(m)+j} \right)}}\right),$$ and $\sum'$ denotes the summation over $p\leq (\log D_j)^\epsilon$ such that any one of the conditions given below  $p\leq (\log D_j)^\epsilon$ does not get satisfied. 
Suppose $\mathbb{Q}\left(\sqrt[l]{\Delta(m)+j} \right)$  satisfy Hypothesis \ref{hypothesis 1} for all $j=1,\dots, k. $
Then there are at least $x^{1/l-\text{o}(1)}$ integers $1 \leq d\leq x$ such that the pure fields $\Q(\sqrt[l]{d+1})$, \dots , $\Q(\sqrt[l]{d+k})$   have arbitrary large class numbers.
% have arbitrary large class numbers
%     for all $j=1,\dots , k. $ 
\end{theorem}
% Then there are atleast $x^{1/l-\text{o}(1)}$ integers $1\leq d\leq x$ such that the class number $h_{\Q(\sqrt[l]{d+j})}$ of the pure cubic field $\Q(\sqrt[l]{d+j})$ satisfy $$h_{\Q(\sqrt[l]{d+j})}\gg_k\frac{\log\log d}{\log^{l-1} d},$$
%     for all $j=1,\dots,k.$
\begin{remark}
    The hypothesis 
   $$\sum_{ \substack{p\leq (\log D_j)^\epsilon, \\ p\equiv \,1(\text{mod}\, \,l) \\\left( \frac{D_j}{\pi}\right)_l=1} }  \frac{l-1}{p}+ \sum_{ \substack{p\leq (\log D_j)^\epsilon, \\ p\equiv 1 \,(\text{mod}\, \, l) \\\left( \frac{D_j}{\pi}\right)_l=1}}' \   \frac{\lambda(p)}{p}     \gg_k \log \log \log x $$  is not vacuous. For instance,  when  $l=3,$ $\lambda(p) = 0$ for $p\equiv 2 \,(\text{mod}\,\, 3)$ (see \cite[Section 2]{DBDY} and proof of \cite[Lemma 9]{Daileda}). This makes  
   $$\sum_{ \substack{p\leq (\log D_j)^\epsilon, \\ p\equiv \,1(\text{mod}\, \,3) \\\left( \frac{D_j}{\pi}\right)_3=1} }  \frac{2}{p}+ \sum_{ \substack{p\leq (\log D_j)^\epsilon, \\ p\equiv 1 \,(\text{mod}\, \, 3) \\\left( \frac{D_j}{\pi}\right)_3=1}}' \   \frac{\lambda(p)}{p} = \sum_{ \substack{p\leq (\log D_j)^\epsilon , \\ p\equiv \,1(\text{mod}\, \,3) \\\left( \frac{D_j}{\pi}\right)_3=1} }  \frac{2}{p}\gg_k \log  \log \log x.$$
   Let us consider $l=5$. We have  $\lambda(p)=\sum_{i=1}^{4}\alpha_i(p)$ (see equation \eqref{allamb}), where  $\alpha_i(p)$ are $20$-th roots of unity (as  $|\text{Gal}(\Q(e^{\frac{2\pi i}{5}},\sqrt[5]{b})/\Q)|=20$)  or $0.$ 
  % Hence $|\lambda(p) +4|\geq \left|\cos({\frac{\pi }{5}})\right|$ unless $\lambda(p)=-4.$ 
   Hence  $\lambda(p)=-4c_p$, where $c_p$ is very close to $1$ is not possible. 
  Based on this argument, we suppose that
  $$\sum_{ \substack{p\leq (\log D_j)^\epsilon, \\ p\equiv \,1(\text{mod}\, \,5) \\\left( \frac{D_j}{\pi}\right)_5=1} }  \frac{4}{p}+ \sum_{ \substack{p\leq (\log D_j)^\epsilon, \\ p\equiv 1 \,(\text{mod}\, \, 5) \\\left( \frac{D_j}{\pi}\right)_5=1}}' \   \frac{\lambda(p)}{p} \gg_k \log \log \log x.$$
   
   % Hence we use the available Landau bound of $\sqrt{D_j} \log^{l-1} D_j$ (see \cite{Siegel}). This weakens our lower bound compared to Theorem \ref{DBDY}.
\end{remark}
The structure of the paper is as follows. In Section \ref{Two propositions}, we prove  Proposition \ref{main prop} which  deals with an approximation of $L(1,\tilde{\pi}_{d_K})$ (see equation \eqref{L(s,pi)}). 
% and Proposition \ref{prop 2} talks about an upper bound for the regulator of the field $\Q\left(\sqrt[l]a\right)$. 
In Section \ref{two}, we prove two lemmas leading to the main theorem's proof. The proof of Theorem \ref{main theorem} is covered in Section \ref{proof of main thm}. 
\section{A Proposition}\label{Two propositions}
 Let $F$ be a number field and $E$ be a degree $n$ extension of $F$ with Galois closure $\hat{E}$.  Let   $G=\text{Gal}(\hat{E}/F)$ and  $\pi$ be a finite dimensional complex representation of $G.$ For a prime ideal $\p$ in $F$, let $I_\p $ and $\sigma_\p$ denote the inertia group and Frobenius element for an ideal $\P$ of $\hat{E}$ lying over $\p.$ Let the space corresponding to $\pi$ is $V$ and $V^{I_\p}$ denote the subspace of $V$ fixed by $I_\p.$
We define $L_\p(s,\hat{E}/F,\pi)=\det\left(I - \pi(\sigma_\p)|_{V^{I_\p}}N(\p)^{-s} \right)^{-1}.$
The Artin L-function associated to $\pi$ is given by 
\begin{equation}\label{euler product}
L(s,\hat{E}/F,\pi)=\prod_{\p}L_\p(s,\hat{E}/F,\pi).
\end{equation}
 If the degree of $\pi$ is ${n'},$ then $L_\p(s,\hat{E}/F,\pi)$ takes the form $$\prod_{i=1}^{{n'}}\left( 1- \alpha_i(\p)N(\p)^{-s} \right)^{-1}$$ where $\alpha_i(\p)$ are either $|G|$-th roots of unity (see \cite[Chapter VII (10.1)]{NNEU} or $0 $ and $N(\mathfrak{p})$ denotes the absolute norm of $\p.$ 
 
 In particular, for $F=\Q,$ on expanding the Euler product given by equation \eqref{euler product} as a Dirichlet series, we have 
 \begin{equation}\label{lamdaprime}
L(s,\hat{E}/F,\pi)=\sum_{m=1}^\infty\frac{\lambda'(m)}{m^s}.
 \end{equation}
 For an unramified rational prime $p$ in $\hat{E},$ ${\lambda'}(p)=\text{Tr}(\pi(\sigma_p)).$
 Hence if $p$ splits completely in  $\hat{E}$, then $\sigma_p $ is trivial and
 \begin{equation}\label{complete splitting}
 {\lambda'}(p)={n'}.
 \end{equation}
 We have \begin{equation}\label{allamb}
     \lambda'(p)=\sum_{i=1}^{n'}\alpha_i(p)
 \end{equation}
  for all rational prime $p $ and $|\lambda'(p)|\leq {n'}.$
 
Let  $H=\text{Gal}(\hat{E}/E).$ Action of $G$ on the coset space $G/H$ gives rise to an $n-$dimensional complex (permutation) representation $\rho$ of $G.$ This representation is induced from the trivial representation of $H$ which implies 
\begin{equation}\label{DeArt}
    L(s,\hat{E}/F,\rho)=L(s,\hat{E}/E,1_H)=\zeta_E(s).
\end{equation}
% For the notations used in the above equation, we mostly refer to section 2.1 and section 3.1 of \cite{Daileda}. 
Let us take $K=\Q(\sqrt[l]{a})$, where $a\in \mathbb{N}$ is $l-$th powerfree, and $d_K$ to be the absolute value of the discriminant of $K$ (see \cite[Theorem 1.1]{JhSa} for an expression for $d_K$).
Let $L$ be the Galois closure of $K$. Then $L=\Q(\omega,\sqrt[l]{a}),$ where $\omega$ is a primitive $l$-th root of unity. Using \cite[Chapter II(9), Problems 4(d)]{PM}, we see that $\text{Gal}(L/\Q)$ is isomorphic to a group of order $(l-1)\times l.$
On keeping equation \eqref{DeArt} in mind and following \cite[equation (20)]{Daileda}, we see that $$\zeta_{K}(s)=\zeta(s) L(s,L/\Q,\pi),$$ where $\rho\cong  1\oplus \pi$ with $\pi$ being a $(l-1)$ dimensional representation of $\text{Gal}(L/\Q)$. 
    By Langlands conjecture (\cite[Conjecture 1.8.1]{Bump}), there is an automorphic representation $\tilde{\pi}_{d_K}$ with conductor $d_K$ such that $ L(s,L/\Q,\pi)=L(s,L/\Q , \tilde{\pi}_{d_K})$(see  \cite[Lemma 7]{Daileda} also). 
The Artin L-function $L(s,L/\Q , \tilde{\pi}_{d_K})$ has a Dirichlet series given by 
\begin{equation}\label{L(s,pi)..}
L(s,L/\Q , \tilde{\pi}_{d_K})=\sum_{m=1}^\infty \frac{\lambda(m)}{m^s}.
\end{equation}
Note that we have taken $\hat{E}=\Q(\omega,\sqrt[l]{a})$, $F=\Q$, $\pi=\tilde{\pi}_{d_K}$ and $\lambda'=\lambda$ in equation \eqref{lamdaprime}.
% Let us consider an entire $L$ function $L(s,L/\Q,\pi)$ where $L/\Q$ is a finite  Galois extension and $\pi $ is an $n-$dimesional complex representation of $\text{Gal}(L/\Q)$ of conductor $N.$
The following approximation of $L(1,\tilde{\pi}_{d_K}):=L(1,L/\Q,\tilde{\pi}_{d_K})$ is useful for the main theorem. 
\begin{Proposition}\label{main prop}
 Let $a$ be a positive $l$-th power-free integer less than or equal to x. Let $K=\Q(\sqrt[l]{a})$ be a pure prime degree field and  $d_K$ be the absolute value of the discriminant of $K$ . Then for any $0<\epsilon<1,
 $
 \begin{equation}\label{L(s,pi)}
 \log L(1,\tilde{\pi}_{d_K})=\sum_{p\leq (\log d_K)^\epsilon} \frac{\lambda(p)}{p}+\text{O}_\epsilon(1),
 \end{equation}
 with at most $\text{O}(x^{\frac{1}{4}})$ exceptions $a\leq x.$ 
\end{Proposition}
\begin{proof}
    Let $Q=x$ and $S(Q)=\{\tilde{\pi}_{d_K}: a\leq x\},$ then by Langlands conjecture, $Q\ll|S(Q)|\ll Q^{1.1}.$  For $\sigma>\frac{1}{2}$ and $T>0,$ let $N(\sigma,T,\tilde{\pi}_{d_K})$ denote the number of zeroes of $L(s,\tilde{\pi}_{d_K})$ in the rectangle $[\sigma,1]\times[-T,T].$ Theorem analogue to \cite[Theorem 5]{Daileda} and \cite[Corollary 3.3]{Daileda} can be obtained for $S(Q). $  By taking $e=1, d=1.1$ and $\sigma=\frac{49}{50}$ in these obtained analogues, we get $(4d+6)(1-\sigma)<\frac{1}{4}$ and  
    \begin{equation}\label{x1/4}
    \sum_{\tilde{\pi}_{d_K}\in S(Q)} N(\sigma,T,\tilde{\pi}_{d_K})\ll Q^{\frac{1}{4}}.
    \end{equation}

    Suppose that,  some $\tilde{\pi'}_{d_K}\in S(Q)$ has more than or equal $\text{O}(x^{\frac{1}{4}+\epsilon'})$ zeroes in $[\sigma,1]\times[-T,T] $ for some $\epsilon'>0,$ then 
    $$\sum_{\tilde{\pi}_{d_K}\in S(Q)} N(\sigma,T,\tilde{\pi}_{d_K})\geq N\left(\sigma,T,\tilde{\pi'}_{d_K}\right)\gg x^{\frac{1}{4}+\epsilon},$$ contradicting equation \eqref{x1/4}.
    Thus for all $\tilde{\pi}_{d_K}\in S(Q)$ with atmost $\text{O}(x^{\frac{1}{4}})$ exception, $L(s,\tilde{\pi}_{d_K})$  is free from zeroes in the rectangle $[\sigma,1]\times[-(\log Q)^2,(\log Q)^2].$ Now using \cite[Proposition 2]{Daileda}, for $0<\epsilon<1\leq \frac{112}{7(1-\sigma)},$ 
    $$
 \log L(s,\tilde{\pi}_{d_K})=\sum_{p\leq (\log d_k)^\epsilon} \frac{\lambda(p)}{p}+\text{O}_\epsilon(1),
 $$
 with at most $\text{O}(x^{\frac{1}{4}})$ exceptions $a \leq x.$  
\end{proof}
\begin{remark}
    Above proposition is a variant of \cite[Proposition 2]{DBDY}, \cite[Proposition 5]{Dukecubic}.  
\end{remark}
\subsection{Upper bound for the regulator of $\mathbb{Q}\left(\sqrt[l]{a}\right)$}\label{sec 2.1}
Let $L'$ be a number field with $[L':\Q]=n.$ Let $\sigma_1, \dots , \sigma_{r_1} $ denote the real embeddings of $L'.$ Let $\sigma_{r_1+1}, \dots, \sigma_{r_1+r_2}$ denote all  pairwise  complex non-conjugate embeddings  of $L'.$ We have $r_1+2r_2=n$ and let $r_1+r_2-1=r'.$  Let $u_1, \dots, u_{r'}$ be  a fundamental system of generators for the unit group of $L$ modulo roots of unity.  
Given $x \in L'^\times, $ define
$$l_{i'}(x) = \begin{cases}
            |\sigma_{i'}(x)|, \text{ if  $1\leq i'\leq r_1$}\\
            |\sigma_{i'}(x)|^2, \text{ if $r_1+1\leq i'\leq r_1+r_2$}.
           \end{cases}
           $$
Consider the  matrix $A_{L'}=(a_{i',j})$ of order $(r+1)\times r,$ where $a_{i',j}=\log( l_{i'}(u_j)).$  The determinant of the matrix obtained by deleting any one of the row  of $A$ is defined to the regulator of $L'$ and denoted by $R_{L'}.$ 

One of the key ingredients of \cite[Proposition 2.2]{DBDY}, which gives us a sharp bound for the regulator of the field of the form $\Q\left(\sqrt[3]{n^3+r} \right)$ with $n,r\in \mathbb{N},$
 is the existence of units of certain form. To be precise, we see that in the proof of \cite[Proposition 2.2]{DBDY} 
uses the following fact: If $r|3n^3$, then $\frac{r}{(\omega-n)^3}$ is a unit in $\Q\left(\sqrt[3]{n^3+r}\right),$ where $\omega=\sqrt[3]{n^3+r}.$ Since $r_1+r_2-1=1$ for both the cases $\Q\left(\sqrt{n^2+r} \right)$ and $\Q\left(\sqrt[3]{n^3+r}\right)$, it is sufficient to find a unit $u$ in each field and bound the regulator (which is asymptotic to $\log u$) by bounding the unit.  For our case of $\Q\left(\sqrt[l]{n^l+r}\right),$ there are $\frac{l-1}{2}$ many fundamental units, all of which can not be written explicitly, making it difficult to get a sharp upper bound for the regulator like \cite[Proposition 2.2]{DBDY}. 
We do, however, have an explicit value of a unit for the field $\Q\left(\sqrt[l]{n^l+r}\right).$  Since $l$ is a prime number and  $r|ln^{l-1},$ from \cite{Stender}  we see that $u=\frac{r}{(\tilde{\omega}-n)^l}$ is a unit in  $\Q\left(\sqrt[l]{n^l+r}\right)$ where $\tilde{\omega}=\sqrt[l]{n^l+r}.$ Following along \cite[Proposition 2.2]{DBDY}, we see that  $u=\text{O}\left((n^l+r)^{l-1}\right).$
We do believe that the other remaining $\frac{l-3}{2}$ units could hopefully be shown to  $\text{O}\left((n^l+r)^{l-1}\right),$ which would help to get a sharp bound of $\text{O}(n^l+r)$ (an analogue of \cite[Proposition 2.2]{DBDY}) for $R_{\Q\left(\sqrt[l]{n^l+r}\right)}$ and remove Hypothesis \ref{hypothesis 1}.
However, at this point,  we do not have much information about other $\frac{l-3}{2}$ units; so we choose to 
state the best known, i.e. Landau's bound (see \cite{Siegel}) on regulators instead. 
% In general, let the  number field $K$ have degree $n$ and absolute value of the discriminant $D.$  Then we have  $R_K=\text{O}(\sqrt{D}\log^{n-1} D)$ . 
% where $D$ denotes the absolute value of the discriminant of 
%  Since we have knowledge of only one possible unit, we stick to Landau's bound for our case. 
\begin{Proposition}[Landau]\label{prop 2}
    Let $l$ be a prime number and $n, r$ be  positive integers. Let $R_{\Q\left(\sqrt[l]{n^l+r}\right)}$ denote the regulator for the pure cubic field $\Q\left(\sqrt[l]{n^l+r}\right)$ and $D_{\Q\left(\sqrt[l]{n^l+r}\right)}$ denote the absolute value of its discriminant.  Then 
    $$
 R_{\Q\left(\sqrt[l]{n^l+r}\right)}=\text{O}\left(\sqrt{D_{\Q\left(\sqrt[l]{n^l+r}\right)}}\log^{l-1}D_{\Q\left(\sqrt[l]{n^l+r}\right)}\right).
    $$
\end{Proposition}
\begin{remark}
   We want the reader to keep Hypothesis \ref{hypothesis 1} in mind, which we will use to derive Theorem \ref{main theorem}. Proving the hypothesis \ref{hypothesis 1} or any sharper version of it might be of interest to the reader.  On a similar note, one might try to get a sharper bound than Landau's bound for some class of number fields.
\end{remark}
\begin{remark}[Comparing regulator bounds for $l=2$,$3$]
     For the case $l=2$ and $3,$ the regulator is $\text{O}(\log D_j)$. However, while dealing with the case $l>3,$ we are unable to find a sharp upper bound for the regulator of fields of the form $\Q(\sqrt[l]{a^l+j}),$ as explained earlier in this subsection. The Landau bound does not help, forcing us to invoke Hypothesis \ref{hypothesis 1}. In the upcoming discussion in this subsection, we provide some computational evidence for the validity of  Hypothesis \ref{hypothesis 1} for  $l=5.$
\end{remark}

\subsection{Computational evidence}\label{comevi}
In the following table, we record some examples supporting Hypothesis \ref{hypothesis 1}.       Let $\lceil x\rceil$ denote the celiling function evaluated at $x.$ Let $\tau(x)$  denotes the number of divisors of $x.$  Given $n\in \N,$ there are $\tau(ln^l)$ many $a$ for which one has to check whether Hypothesis \ref{hypothesis 1} holds true. Computing the regulator for a number field is a challenging task, even for modern-day computers. Hence, we stick to $l=5.$ One needs to check for regulators for $\tau(5n^5)$ many fields given by $\Q(\sqrt[5]{a})$, where $a=n^5+r$ and $r|5n^5.$

First we check for $n=1$ to $5$ and compute the ratio $\alpha(a):=\frac{\left\lceil R_{\Q\left(\sqrt[5]{a}\right)}\right\rceil}{\left\lfloor \sqrt{D_{\Q\left(\sqrt[5]{a}\right)}} \right\rfloor} $ where $a$ is as per the last paragraph.  We need to check the ratio for a total of $55$ fields. It took more than $3$ hours in a PC with specification given by \textbf{11th Gen Intel(R) Core(TM) i9-11900 @ 2.50GHz (16 CPUs)} using SAGE  to check and verify that the ratio is less than $1$ for these $55$ fields. It took more than two days to verify the hypothesis \ref{hypothesis 1} for $n=8$. It becomes practically impossible to check for, say all $n $ from $6$ to $100.$
Due to an explicit expression for discriminant \cite{JhSa}, computing discriminant is easy for  $n=6$ to $100.$ Thus the challenge is to primarily compute the regulator, for which  
SAGE implemented efficient algorithm seems to be missing in the literature. 
In the following tables, we record some examples obtained in  computation.
\begin{table}[ht]
\caption{$n=5$}
\begin{center}
\begin{tabular}{|c|c|c|c|}  \hline
  Field & $\lceil \text{Regulator}\rceil$  &   $\left\lfloor\sqrt{\text{Discriminant}}\right\rfloor$  &  $    \lceil \log_{10} \alpha(a) \rceil$ \\
  \hline
 $\Q(\sqrt[5]{6} ) $ &   $2$      & $ 2012$   &      $ -1$  \\
  $\Q(\sqrt[5]{10} ) $ &   $195$      & $  5590$   &      $ -1$  \\
  $\Q(\sqrt[5]{ 30} ) $ &   $ 1348$      & $  50311$   &      $ -1$  \\
  $\Q(\sqrt[5]{  130} ) $ &   $24127$      & $ 944738$   &      $ -1$  \\
  $\Q(\sqrt[5]{  630} ) $ &   $ 10449$      & $ 2465264$   &      $ -2$  \\
  $\Q(\sqrt[5]{ 15630} ) $ &   $  127544113$      & $ 13656611877$   &      $ -2$  \\
 
    \hline
\end{tabular}
\end{center}
\end{table}
\begin{table}[ht]
\caption{$n=8$}
\begin{center}
\begin{tabular}{|c|c|c|c|c|c|c|c|}  \hline
  Field & $\lceil \text{Regulator}\rceil$  &   $\left\lfloor\sqrt{\text{Disc.}}\right\rfloor$  &  $    \lceil \log_{10} \alpha(a) \rceil$&  Field & $\lceil \text{Regulator}\rceil$  &   $\left\lfloor\sqrt{\text{Disc.}}\right\rfloor$  &  $    \lceil \log_{10} \alpha(a) \rceil$ \\
  \hline
 $\Q(\sqrt[5]{9 } ) $ &   $13 $      & $ 503 $   &      $ -1 $ &  $\Q(\sqrt[5]{ 10} ) $ &   $ 195$      & $ 5590 $   &      $-1  $   \\
  $\Q(\sqrt[5]{12 } ) $ &   $67$      & $ 2012 $   &      $  -1$ & $\Q(\sqrt[5]{13 } ) $ &   $59 $      & $ 9447 $   &      $ -2 $  \\
  $\Q(\sqrt[5]{16  } ) $ &   $ 5 $      & $ 223  $   &      $ -1 $ & $\Q(\sqrt[5]{18 } ) $ &   $ 11$      & $ 402 $   &      $-1  $  \\
  $\Q(\sqrt[5]{  24 } ) $ &   $ 11$      & $402 $   &      $ -1 $& $\Q(\sqrt[5]{28 } ) $ &   $236 $      & $ 10957 $   &      $ -1 $   \\
  $\Q(\sqrt[5]{ 40 } ) $ &   $73  $      & $ 5590$   &      $ -1 $ & $\Q(\sqrt[5]{48 } ) $ &   $50 $      & $2012  $   &      $ -1 $  \\
  $\Q(\sqrt[5]{ 72  } ) $ &   $ 50  $      & $ 2012 $   &      $-1 $ & $\Q(\sqrt[5]{88 } ) $ &   $32 $     & $27056  $   &      $ -2 $  \\
  $\Q(\sqrt[5]{ 136  } ) $ &   $1444   $      & $ 64622 $   &      $ -1$ & $\Q(\sqrt[5]{168 } ) $ &   $566 $     & $ 19722 $   &      $-1  $  \\
    $\Q(\sqrt[5]{ 264  } ) $ &   $1264   $      & $ 243507 $   &      $-2 $ & $\Q(\sqrt[5]{ 328} ) $ &   $ 831$     & $ 375883 $   &      $ -2 $  \\
      $\Q(\sqrt[5]{ 520  } ) $ &   $  4449 $      & $ 944738 $   &      $ -2$ & $\Q(\sqrt[5]{ 648} ) $ &   $67 $     & $2012  $   &      $ -1 $  \\
        $\Q(\sqrt[5]{ 1032  } ) $ &   $ 19717  $      & $ 744209 $   &      $-1 $ & $\Q(\sqrt[5]{1288 } ) $ &   $38249 $     & $ 5796111 $   &      $ -2 $  \\
          $\Q(\sqrt[5]{ 2056  } ) $ &   $ 551977  $      & $14769005  $   &      $ -1$ & $\Q(\sqrt[5]{2568 } ) $ &   $ 155545$     & $ 4608133 $   &      $-1  $  \\
            $\Q(\sqrt[5]{ 4104  } ) $ &   $ 989  $      & $ 726498 $   &      $ -2$ & $\Q(\sqrt[5]{ 5128} ) $ &   $ 22156$     & $ 91875784 $   &      $  -3$  \\
              $\Q(\sqrt[5]{ 8200  } ) $ &   $ 40917  $      & $ 9397075 $   &      $-2 $ & $\Q(\sqrt[5]{10248 } ) $ &   $ 324012$     & $ 366930034 $   &      $ -3 $  \\
                $\Q(\sqrt[5]{ 16392  } ) $ &   $  193082 $      & $  938791003$   &      $-3 $ & $\Q(\sqrt[5]{20488 } ) $ &   $ 6464945$     & $ 1466574600 $   &      $ -2 $  \\
                  $\Q(\sqrt[5]{ 32776  } ) $ &   $ 103219  $      & $750666311  $   &      $-3 $ & $\Q(\sqrt[5]{ 40968} ) $ &   $214882 $     & $ 130311288 $   &      $ -2 $  \\
                    $\Q(\sqrt[5]{ 81928  } ) $ &   $  579552 $      & $478601058  $   &      $ -2$ & $\Q(\sqrt[5]{163848 } ) $ &   $ 450677433$     & $ 93796647781 $   &      $  -2$  \\
 
    \hline
\end{tabular}
\end{center}
\end{table}
 \section{Two  lemmas}\label{two}
In this section, we prove two main lemmas analogous to Lemma 3.1 and Lemma 3.2 in \cite{DBDY}. 
These two lemmas are
crucial to prove Theorem \ref{main theorem}. 
First, we consider some notation involving $k,l.$. Recall that $ l$ is a fixed odd prime number greater than or equal to $3$. 
Let $k$ be a fixed natural number. Let $P=\text{lcm}(1,\, \dots \, , k)$, $\Delta(m)=(mP)^l,$  and $P_j=\frac{P}{j}.$  If  $\Delta(m)\leq x,$  then $m\leq \frac{x^{1/l}}{P}.$  Let  $M=\frac{x^{1/l}}{P}$ and $$q=\prod_{k<p\leq (\log M)^\epsilon} p,$$ where $p$ is a rational prime. 
\begin{lemma}\label{lemma a}
    Let $0< \epsilon <1$ be given and $D_j$ be the absolute value of the discriminant of the
pure field $\mathbb{Q}\left(\sqrt[l]{\Delta(m)+j} \right)$. Then for all $j=1,\dots, k, $ there exists $m_0 ( \text{mod} \,\,q)$ such that if $m \equiv m_0$ (mod
$q)$, then for all primes $p \equiv 1 (\text{mod}$ $l)$ with $l^{2k} (lk + 1)^2 < p \leq  (\log M)^\epsilon,$  we have
$\left( \frac{D_j}{\pi} \right)_l=1$
where $p = \pi \bar{\pi}$ in $Q( \sqrt{-l })$ and $(\frac{\cdot}{\cdot})_l$ is the $l$th power residue  symbol (see \cite[Chapter VI, Exercise 9]{FrTa}). Moreover, we can
take $m_0\equiv 0 $(mod $p)$ for all primes $p$ with $k <p \leq  l^{2k} (lk + 1)^2$.
\end{lemma}
\begin{proof}
Let $p$ be a prime $\equiv 1 (\text{mod}$ $l)$ with $k< p\leq (\log M)^\epsilon$. Note that $P$ is invertible modulo $p>k.$  Hence there are $l$ times as many nonzero classes of $m_0$ (mod $p$) satisfying 
\begin{equation}\label{equation 1 lemma a}
    \left( \frac{(m_0P)^l+1}{\pi} \right)_l=\left( \frac{(m_0P)^l+2}{\pi} \right)_l=\dots =\left( \frac{(m_0P)^l+k}{\pi} \right)_l=1
\end{equation}
as the number of $n\equiv (m_oP)^l$ modulo $p$ with 
\begin{equation}\label{equation 2 lemma a}
    \left( \frac{n}{\pi} \right)_l=\left( \frac{n+1}{\pi} \right)_l=\dots =\left( \frac{n+k}{\pi} \right)_l=1.
\end{equation}
For $j=0, \dots, k, $ let $r_j \in \{0,1, \dots, l-1\},$ and consider the polynomial $$
Q_{(r_0,\dots, r_k)}(X)=\prod_{j=0}^k(X+j)^{r_j}.
$$  On applying \cite[Theorem 11.23]{IK} for each $(r_0,\dots, r_k)\neq (0,\dots ,0),$ for the sum 
$$
S_{(r_0,\dots, r_k)} =\sum_{n=1}^p \left( \frac{Q_{(r_0,\dots, r_k)}(n) }{\pi}\right)_l, 
$$ we obtain $|S_{(r_0,\dots, r_k)}|\leq k \sqrt{p}.$ Hence the number of solutions modulo $p$ to the equation \eqref{equation 2 lemma a} is 
\begin{align*}
&\frac{1}{l^{k+1}} \sum_{n=1}^{p-k-1}\sum_{j=0}^{k} \left(  1+ \left( \frac{n+j}{\pi} \right)_l  +\left( \frac{n+j}{\pi} \right)^2_l +\dots \left( \frac{n+j}{\pi} \right)^{l-1}_l  \right)
\\
&=\frac{1}{l^{k+1}} \sum_{n=1}^{p}\sum_{j=0}^{k} \sum_{t=0}^{l-1} \left( \frac{n+j}{\pi} \right)^t_l+O(k+1)\\
&=\frac{p}{l^{k+1}} +\frac{1}{l^{k+1}}   \sum_{(r_0,\dots, r_k)\neq (0,\dots, 0)} S_{(r_0,\dots, r_k)}+O(k+1)\\
&=\frac{p}{l^{k+1}} +O\left(k\sqrt{p}+k+1 \right).
\end{align*}
Therefore there is atleast one $m_0$  (mod $p$) satisfying equation \eqref{equation 1 lemma a} provided $ p \geq  l^{2k} (lk + 1)^2$. Using the Chinese remainder theorem,  we can obtain at least one
residue class (mod $q$). 
\end{proof}
\begin{lemma}\label{lemma b}
    Let $m_0$(mod $q)$ be as given by Lemma \ref{lemma a}. For $M^{1-o(1)}$ integers with $1 \leq m \leq 
M$ and  $m \equiv m_0$ (mod
$q)$, we have
$$
D_j\gg_k
\Delta(m),$$
for all  $j=1,\dots, k$.
\end{lemma}
\begin{proof}
    Let $S_j^l$ and $s_j^l$ be the largest $l$-th power dividing $\Delta(m)+j$ and $j$ respectively.   Using \cite[Theorem 1.1]{JhSa}, we have $$D_j=l^{l-2} \prod_{r|(\Delta(m)+j)} r^{l-1}$$ or $l^{l} \prod_{r|(\Delta(m)+j)} r^{l-1}. $
 For the case $S_j=s_j$, we get  $$D_j\geq \frac{\Delta(m)+j}{S_j^l} = \frac{\Delta(m)+j}{s_j^l}\geq \frac{\Delta(m)}{k^l}.$$ Hence  we will show that for $M^{1-o(1)}$ integers with $1 \leq m \leq 
M$ and  $m \equiv m_0$ (mod
$q),$ we have $S_j=s_j.$
 On letting $F_j(m)=j^{l-1} (mP_j)^l+1,$ we get $\Delta(m)+j=jF_j(m).$ Thus $S_j=s_j$ if and only if $F_j(m)$ is $l$-th power free. We prove that there are $M^{1-o(1)}$  integers satisfying
$1 \leq m \leq M$ with $m\equiv m_0$ (mod $q)$, $F_j(m)$ is $l$-th power free.

Following along the proof of \cite[Lemma 2.2]{CFGKY},  for $m \equiv m_0$ (mod
$q)$ and  $1\leq p \leq (\log M)^{\epsilon}$, $p$ does not divide $F_j(m).$ Let $z=q^l(\log M)^ {kl^2}.$ Using similar sieving argument as in the proof of \cite[Lemma 2.2]{CFGKY},  for $\gg_k M^{1-o(1)}$ many integers with $m \equiv m_0$ (mod
$q) $ and for all $j=1,\dots, k,$  we have the following:
\begin{itemize}
    \item $p$ does not divide $F_j(m)$ for all prime $ (\log M)^{\epsilon}<p \leq z,$
    \item $p^l$ does not divide $F_j(m)$ for all primes $p$ satisfying 
    $z<p\leq \frac{2MP}{z^{1/l}}.$
\end{itemize}
Let $m$ be an integer satisfying the above two conditions with $p^l|F_j(m)$ for some prime $p $ and $j, $ then $z> \frac{2MP}{z^{1/l}}.$ Hence if $F_j(m)=tp^l,$ then 
$$t=\frac{F_j(m)}{p^l}\leq \frac{(mP)^l+j}{p^l}\leq \frac{2(MP)^l}{(2MP/z^{1/l})^l}\leq z.$$
This forces $t=1$ as $F_j(m)$ is not divisible by any prime $p\leq z. $ Thus $F_j(m)=p^l$ and $ m$ is a solution of the equation 
$1=p^l -(j^{l-1} P_j^l) m^l. $ For a fixed $j, $ using \cite[Theorem 1.1]{BW}, we see that there is at most one solution $(p,m)$ possible to the equation. 
After discarding at most one such integer $m$, the proof of the lemma is complete. 
 \end{proof}
\section{Proof of Main theorem}\label{proof of main thm}
We now give a proof of Theorem \ref{main theorem}.
\begin{proof}
    Using Lemma \ref{lemma a} and Lemma \ref{lemma b}, there are $x^{\frac{1}{l}-\text{o}(1)}$ integers $1\leq \Delta(m)\leq x$ such that $D_j\gg_k \Delta(m)$ and $\left( \frac{D_j}{\pi}\right)_l=1$, for all primes $p\equiv 1 $(mod $l$) and $l^{2k} (lk + 1)^2 < p \leq  (\log M)^\epsilon$ and for all $j=1,\dots,k.$
    Using  Proposition \ref{main prop},with at most $\text{O}(x^{\frac{1}{4}})$ exceptions,  we see that $$
    \log L(1,\tilde{\pi}_{D_j})=\sum_{p\leq (\log D_j)^\epsilon} \frac{\lambda(p)}{p}+\text{O}_\epsilon(1). 
    $$

   If $p\equiv 1(\text{mod}\, \, l)$ and $\left( \frac{D_j}{\pi}\right)_l=1,$ then $p$ splits completely in $\Q(\sqrt[l]{\Delta(m)+j}).$   From equation \eqref{complete splitting} we get  $\lambda(p)=l-1$ for complete splitting.  
   %On applying 
 % Chebotarev's density theorem, we note that $$ 
 % \sum_{ \substack{p\leq x, \\ p\equiv 1(\text{mod}\, l) \\\left( \frac{D_j}{\pi}\right)_l=1} }  \frac{l-1}{p}\gg_k \log \log x.$$
Using
  $$\sum_{ \substack{p\leq (\log D_j)^\epsilon, \\ p\equiv 1(\text{mod}\, l) \\\left( \frac{D_j}{\pi}\right)_l=1} }  \frac{l-1}{p}+ \sum_{ \substack{p\leq (\log D_j)^\epsilon, \\ p\equiv 1(\text{mod}\, l) \\\left( \frac{D_j}{\pi}\right)_l=1}}' \   \frac{\lambda(p)}{p}     \gg_k \log \log \log x     ,$$
  gives us $\log L(1,\tilde{\pi}_{D_j})\gg_k \log \log \log x.$
 
    It now follows from the class number formula that 
 \begin{equation}\label{main eqn}
    h_{\Q\left(\sqrt[l]{\Delta(m)+j}\right)}=\frac{\sqrt{D_j} L(1,\tilde{\pi}_{D_j})}{R_{\Q\left(\sqrt[l]{\Delta(m)+j}\right)}}\gg_k \frac{\sqrt{\Delta(m)} \log\log \Delta(m)}{R_{\Q\left(\sqrt[l]{\Delta(m)+j}\right)}},
    \end{equation}
for all $j=1,\dots, k. $
The proof is now complete by using Hypothesis \ref{hypothesis 1} on the regulator of $\Q\left(\sqrt[l]{\Delta(m)+j}\right)$ for all $j$. 
\end{proof}
% $$h_{\Q\left(\sqrt[l]{\Delta(m)+j}\right)}=\frac{\sqrt{D_j} L(1,\tilde{\pi}_{D_j})}{R_{\Q\left(\sqrt[l]{\Delta(m)+j}\right)}}\gg_k \frac{\sqrt{\Delta(m)} L(1,\tilde{\pi}_{D_j})}{\sqrt{\Delta(m)} \log^{l-1} \Delta(m)}\gg_k \frac{\log\log \Delta(m)}{\log^{l-1} \Delta(m)},$$
\begin{remark}[Necessity of Hypotheis \ref{hypothesis 1}]\label{Nechyp}
   Using  Proposition \ref{prop 2} in equation \eqref{main eqn},  we get $$ h_{\Q\left(\sqrt[l]{\Delta(m)+j}\right)}\gg_k \frac{\log\log \Delta(m)}{\log^{l-1} \Delta(m)},$$ which is a triviality since $\frac{\log\log \Delta(m)}{\log^{l-1} \Delta(m)}\rightarrow 0$ as $\Delta(m)\rightarrow \infty. $
  This implies Proposition \ref{prop 2}(Landau's bound) will not suffice to ensure an arbitrarily large class number. 
Thus, assuming Hypothesis \ref{hypothesis 1} is necessary.  For completeness, we note that we need an upper bound that is at least a factor reduction of $\frac{\log \log D_j}{\log^{l-1}D_j}$  in Landau bound. 
\end{remark}
% \begin{corollary}
%     Let $k$ be a fixed positive integer.   Suppose  $R_{\Q\left(\sqrt[l]a\right)}=\text{o}\left(  \sqrt{D_{\Q\left(\sqrt[l]a\right)}}\log\log D_{\Q\left(\sqrt[l]a\right)}\right) $ for any positive integer $a$. Then there are at least $x^{1/l-\text{o}(1)}$ integers $d\leq x$ such that the pure fields $\Q(\sqrt[l]{d+1})$, \dots , $\Q(\sqrt[l]{d+k})$ have arbitrary large class numbers
%     for all $j=1,\dots , k. $ 
% \end{corollary}
% \begin{proof}
%     The proof is almost similar to the proof of Theorem \ref{main theorem}. We only need to focus on the equation \eqref{main eqn}, where the regulator is involved. The result then follows from the fact that as $d \rightarrow \infty$, $\frac{\sqrt{\Delta(m)}\log\log \Delta(m)}{f(\Delta(m))}\rightarrow \infty,$ where $f(\Delta(m))=\text{o}(\sqrt{\Delta(m)}\log\log \Delta(m)).$
% \end{proof}
\section*{Acknowledgement}
We would like to thank IISER Thiruvananthapuram for providing an ideal working environment. 
We thank Subham Bhakta for his valuable input on the article. SK’s research was supported by SERB grant CRG/2023/009035.
JD acknowledges the support of the SERB India-funded project ‘Class Numbers of Number Fields’ under SK. JD is grateful to  COEP Technological University, Pune, for an excellent working environment.  
\bibliographystyle{alpha}
\bibliography{Sato-TateNT.bib}
\end{document}

%% file: format.tex
\newtheorem{thm}{Theorem}[section]
\newtheorem*{thm*}{Theorem}

\newtheorem{lemma}[thm]{Lemma}

\newcommand{\beq}{\begin{equation}}
\newcommand{\eeq}{\end{equation}}